\documentclass[12pt]{amsart}
\usepackage{amsmath,amssymb,amsfonts,a4wide,color,mathtools,url}
\usepackage[initials]{amsrefs}
\usepackage{cite}
\usepackage[utf8]{inputenc} 

\let\phi\varphi
\let\epsilon\varepsilon
\let\setminus\smallsetminus

\let\ndiv\nmid

\let\leq\leqslant
\let\geq\geqslant

\newtheorem{Thm}{Theorem}[section]
\newtheorem{Prop}[Thm]{Proposition}
\newtheorem{Lem}[Thm]{Lemma}

\newtheorem{Rem}[Thm]{Remark}

\numberwithin{equation}{section}

\def\qed{{\hskip0pt\unskip\unskip\nobreak\hfil\penalty50
          \hskip1em\hbox{}\nobreak\hfil
           {$\square$}
          \parfillskip=0pt\finalhyphendemerits=0
          \par}\medskip}

               {\noindent{\bf Proof.}\ }
               {\qed}

               {\noindent{\bf Proof of #1.}\ }
               {\qed}

\begin{document}

\title{Computable Absolutely Normal Numbers and Discrepancies}

\author{
Adrian-Maria Scheerer
}

\thanks{TU Graz, \tt scheerer@math.tugraz.at}

\begin{abstract}
We analyze algorithms that output absolutely normal numbers digit-by-digit with respect to quality of convergence to normality of the output, measured by the discrepancy.
We consider explicit variants of algorithms by Sierpinski, by Turing and an adaption of constructive work on normal numbers by Schmidt. There seems to be a trade-off between the complexity of the algorithm and the speed of convergence to normality of the output. 
\end{abstract}

\date{\today}


\maketitle

\setcounter{tocdepth}{1}



\section{Introduction}
A real number is \emph{normal} to an integer base $b\geq 2$ if in its expansion to that base all possible finite blocks of digits appear with the same asymptotic frequency. A real number is \emph{absolutely normal} if it is normal to every integer base $b\geq 2$. While the construction of numbers normal to one base has been very successful, no construction of an absolutely normal number by concatenation of blocks of digits is known. However, there are a number of algorithms that output an absolutely normal number digit-by-digit. In this work, we analyze some these algorithms with respect to the quality of convergence to normality of the output, measured by the discrepancy.
\medskip

The \emph{discrepancy} of a sequence $(x_n)_{n\geq 1}$ of real numbers is the quantity
\begin{equation*}
D_N(x_n) = \sup_{I \subset [0,1)} \left\vert \frac{\sharp \{1 \leq n \leq N \mid x_n \bmod 1 \in I\}}{N} - \vert I \vert \right\vert,
\end{equation*}
where the supremum is over all subintervals of the unit interval. A sequence is \emph{uniformly distributed modulo one}, or equidistributed, if its discrepancy tends to zero as $N$ tends to infinity.

The \emph{speed of convergence to normality} of a real number $x$ (to some integer base $b\geq 2$) is the discrepancy of the sequence $\{b^n x\}$. $x$ is normal to base $b$ if and only if $\{b^n x \}$ is uniformly distributed modulo one \cite{wall1950:thesis}. Consequently, $x$ is absolutely normal if and only if the orbits of $x$ under the multiplication by $b$ map are uniformly distributed modulo one for every integer $b\geq 2$. It is thus natural to study the discrepancy of these sequences quantitatively as a measure for quality of convergence to normality.

A result by Schmidt \cite{schmidt1972:irreg_of_distr_7} shows that the discrepancy of a general sequence can at most be as good as $O(\frac{\log N}{N})$ when $N$ tends to infinity. The study of such so-called low-discrepancy sequences is a field in its own right. It is an open problem to give a construction of a normal number to some base that attains discrepancy this low. The best result in this direction is due to Levin \cite{levin1999:discrepancy} who constructed a number normal to one base with discrepancy $O(\frac{\log^2 N}{N})$. It is known \cite{gaalgal1964:discrepancy}, that Lebesgue almost all numbers satisfy $D_N = O( \frac{\sqrt{\log \log N}}{N^{1/2}})$.
For more on normal numbers, discrepancies and uniform distribution modulo one see the books \cite{bugeaud2012distribution}, \cite{drmota_tichy} and \cite{kuipers_niederreiter}.
\medskip

A construction for absolutely normal numbers was given by Levin \cite{levin1979:absolutely_normal} where he constructs a real number normal to a specified countable set of real bases larger than one, such that the discrepancy to any one of the bases is $O(\frac{(\log N)^2 \omega(N)}{N^{1/2}})$, where the speed of $\omega \to \infty$ and the implied constant depend on the base. Recently, Alvarez and Becher \cite{alvarez_becher2015:levin} analyzed Levin's work with respect to computability and discrepancy. They show that Levin's construction can yield a computable absolutely normal number $\alpha$ with discrepancy $O(\frac{(\log N)^3}{N^{1/2}})$. To output the first $N$ digits of $\alpha$, Levin's algorithm takes exponentially many (expensive) mathematical operations. Alvarez and Becher also experimented with small modifications of the algorithm.
\medskip

In this work the following algorithms are investigated.

\subsubsection*{Sierpinski} Borel's original proof \cite{borel1909} that Lebesgue almost all numbers are absolutely normal is not constructive. Sierpinski \cite{sierpinski1917:borel_elementaire} gave a constructive proof of this fact. Becher and Figueira \cite{becher_figueira:2002} gave a recursive reformulation of Sierpinski's construction. Sierpinski's algorithm outputs the digits to some specified base $b$ of an absolutely normal number $\nu$, depending on $b$, in double exponential time. $\nu$ has discrepancy $O(\frac{1}{N^{1/6}})$. This short calculation is presented in Section \ref{Sierpinksi_section}.

\subsubsection*{Turing} Alan Turing gave a computable construction to show that Lebesgue almost all real numbers are absolutely normal. His construction remained unpublished and appeared first in his collected works \cite{turing1992:collected_works}. Becher, Figueira and Picchi  \cite{becher_figueira_picchi2007:turing_unpublished} completed his manuscript and show that Turing's algorithm computes the digits of an absolutely normal number in double exponential time. This number has discrepancy $O(\frac{1}{N^{1/16}})$, see Section \ref{Turing_sec}.

\subsubsection*{Schmidt} In \cite{schmidt1961:uber_die_normalitat}, Schmidt gave  an algorithmic proof that there exist uncountably many real numbers normal to all bases in a given set $R$ and not normal to all bases in a set $S$ where $R$ and $S$ are such that elements of $R$ are multiplicatively independent of elements of $S$ and such that $R \cup S = \mathbb{N}_{\geq 2}$. In his construction he requires $S$ to be non-empty. However, in a final remark he points out that it should be possible to modify his construction for $S$ non-empty. The main purpose of this paper is to carry out the details of his remark explicitly to give an algorithmic construction of an absolutely normal number $\xi$. It takes exponentially many (expensive) mathematical operations to output the digits of $\xi$ and the discrepancy of $\xi$ is $O(\frac{\log \log N}{\log N})$. A small modification of the algorithm allows for discrepancy $O(\frac{1}{(\log N)^A})$ for any fixed $A>0$, but the output (i.e. $\xi$) depends on $A$.

Schmidt's main tool is cancellation in a certain trigonometric sum related to multiplicatively independent bases (Hilfssatz 5 in \cite{schmidt1961:uber_die_normalitat} and Lemma \ref{Schmidt:HS5} here). In Lemma \ref{HS5_explicit} we make the involved constants explicit which might also be of independent interest.
\medskip

Becher, Heiber, Slaman \cite{becher_heiber_slaman2013:polynomial_time_algorithm} gave an algorithm that can compute the digits of an absolutely normal number in polynomial time. The discrepancy was not analyzed. In \cite{madritsch_scheerer_tichy2015:absolutely_pisot_normal} it is shown that the discrepancy is slightly worse than $O(\frac{1}{\log N})$, and that at a small loss of computational speed the discrepancy can in fact be $O(\frac{1}{\log N})$.

\subsubsection*{Notation} For a real number $x$, we denote by $\lfloor x \rfloor$ the largest integer not exceeding $x$. The fractional part of $x$ is denoted as $\{x\}$, hence $x = \lfloor x \rfloor + \{x\}$. 
Two functions $f$ and $g$ are $f = O(g)$ or equivalently $f \ll g$ if there is a $x_0$ and a positive constant $C$ such that $f(x) \leq C g(x)$ for all $x\geq x_0$. We mean $\lim_{x\to\infty} f(x)/g(x) = 1$ when we say $f\sim g$ and $g\neq 0$.
We abbreviate $e(x) = \exp(2\pi i x)$.
Two integers $r$, $s$ are multiplicatively dependent, $r\sim s$, when they are rational powers of each other.

In our terminology, \emph{mathematical operations} include addition, subtraction, multiplication, division, comparison, exponentiation and logarithm. \emph{Elementary operations} take a fixed amount of time. 
When we include the evaluation of a complex number of the form $\exp(2\pi i x)$ as a mathematical operation we refer to it as being \emph{expensive}.

\section{Schmidt's Algorithm}

In this section we present an algorithm to compute an absolutely normal number that can be derived from Schmidt's work \cite{schmidt1961:uber_die_normalitat}.
Schmidt's construction employs Weyl's criterion for uniform distribution and as such uses exponential sums. The following estimate for trigonometric series is his main tool.

\begin{Lem}[Hilfssatz 5 in \cite{schmidt1961:uber_die_normalitat}]\label{Schmidt:HS5}
Let $r$ and $s$ be integers greater than $1$ such that $r\not\sim s$. Let $K, l$ be positive integers such that $l \geq s^K$. Then
\begin{equation}
\sum_{n=0}^{N-1} \prod_{k=K+1}^\infty \vert \cos(\pi r^n l /s^k)\vert \leq 2N^{1-a_{20}} \label{HS5:sum}
\end{equation}
for some positive constant $a_{20}$ only dependent on $r$ and $s$.
\end{Lem}

In Section \ref{section:constants} we give an explicit version of Lemma \ref{Schmidt:HS5}.

\subsection{The Algorithm}

We begin by stating Schmidt's algorithm. In Schmidt's notation we are specializing to the case $R = \mathbb{N}_{\geq 2}$. We have incorporated his suggestion how to modify the set $S$ accordingly as to produce absolutely normal numbers.

\subsubsection*{Setup}
Let $R = (r_i)_{i\geq1} = \mathbb{N}_{\geq 2}$ (in increasing order) and let $S = (s_j)_{j\geq 1}$ be a sequence of integers $s > 2$ such that $s_m \leq ms_1$ and such that for each $r \in R$ there is an index $m_0(r)$ such that $r\not\sim s_m$ for all $m \geq m_0(r)$. Let $\beta_{i,j} = a_{20}(r_i, s_j)$ from \ref{HS5:sum} and denote by $\beta_k = \min_{1 \leq i,j \leq k} \beta_{i,j}$. We can assume that $\beta_k < \frac{1}{2}$. Let $\gamma_k = \max(r_1, \ldots, r_k, s_1, \ldots, s_k)$.

Schmidt assumes that the sequences $R$ and $S$ are such that $\beta_k \geq \beta_1 / k^{1/4}$ and that $\gamma_k \leq \gamma_1 k$ holds. This can be achieved by repeating the values of the sequences $R$ and $S$ sufficiently many times. 
Set $\phi(1) = 1$ and let $\phi(k)$ be the largest integer $\phi$ such that the conditions
\begin{equation*}
\phi \leq \phi(k-1) + 1, \quad 
\beta_\phi \geq \frac{\beta_1}{k^{1/4}} \quad \text{and} \quad
\gamma_\phi \leq \gamma_1 k
\end{equation*}
hold. Then modify the sequences $R$ and $S$ according to $r'_i = r_{\phi(i)}$, $s'_i = s_{\phi(i)}$. 
Note that (up to suitable repetition) $S$ can be chosen to be the set of positive integers bigger than $2$ that are not perfect powers.
In principle, using the explicit version of Hilfssatz 5, Lemma \ref{HS5_explicit}, one could write down $R$ and $S$ explicitly.

Following Schmidt, we introduce the following symbols where $m$ is a positive integer. Let $\langle m \rangle = \lfloor e^{\sqrt{m}} + 2s_1 m^3 \rfloor$, denote $\langle m ; x \rangle = \lfloor \langle m \rangle / \log x \rfloor$ for $x>1$ and let $a_m = \langle m ; s_m \rangle$, $b_m = \langle m+1; s_m \rangle$.

\subsubsection*{Algorithm}
Step $0$: Put $\xi_0 = 0$.

Step $m$: Compute $a_m, b_m$, $s_m$. We have from the previous step $\xi_{m-1}$. 
Let $\sigma_m(\xi_{m-1})$ be the set of all numbers
$$\eta_m(\xi_{m-1}) + c_{a_m+1}^{s_m} s_m^{-a_m-1} + \ldots + c_{b_m-2}^{s_m} s_m^{-b_m+2}$$
where the digits $c$ are $0$ or $1$, and where $\eta_m(\xi_{m-1})$ is the smallest of the numbers $\eta = g s_m^{-a_m}$, $g$ an integer, that satisfy $\xi_{m-1} \leq \eta$.

Let $\xi_m$ be the smallest of the numbers in $\sigma_m(\xi_{m-1})$ that minimize
\begin{equation}
A_m'(x) = \sum_{\substack{t=-m \\ t\neq 0}}^m \sum_{\substack{i \leq m\\ m_0(r_i) \leq m}} \left\vert \sum_{j = \langle m; r_i \rangle + 1}^{\langle m+1 ; r_i \rangle} e(r_i^j t x) \right\vert^2
\end{equation}\label{Amprime}
\medskip

The following lemma establishes cancellation in the sums $A'_m$ in order for Weyl's criterion to apply.

\begin{Lem}\label{A_m'_cancellation_lemma}
There exists a positive absolute constant $\delta'_1$ such that
\begin{equation}\label{A_m'estimate}
A_m'(\xi) \leq \delta_1' m^2( \langle m +1 \rangle - \langle m \rangle )^{2-\beta_m}
\end{equation}
\end{Lem}

\begin{proof}
Schmidt's proof of Hilfssatz 7 in \cite{schmidt1961:uber_die_normalitat} can directly by adopted. The inner sum in $A_m'$ over $j$ is essentially the same as in Schmidt's function $A_m$. The outer sums over $r_i$ and $t$ are evaluated trivially and contribute a constant factor times $m^2$. 
\end{proof}

\begin{Rem}
Following the constants in Schmidt's argument shows that $\delta_1' \approx 36$ is admissible.
\end{Rem}

Schmidt shows that the sequence $(\xi_m)_{m\geq 1}$ has a limit $\xi$ that is normal to all bases in the set $R$, i.e. absolutely normal. We have the approximations
\begin{equation}\label{Error_xi_xi_m}
\xi_m \leq \xi < \xi_m + s_m^{-b_m+2}.
\end{equation}

\subsection{Complexity}

We given an estimate for the number of (expensive) mathematical operations Schmidt's algorithm takes to compute the first $N$ digits of the absolutely normal number $\xi$ to some given base $r \geq 2$.

Note that from \ref{Error_xi_xi_m}, in step $M$ of the algorithm the first $b_M -2$ digits of $\xi$ to base $s_M$ are known. They determine the first $(b_M - 2) \frac{\log s_M}{\log r}$ digits of $\xi$ to base $r$. 
Since $s_M \geq 3$ and $b_M \sim e^{M^{1/2}}$, we have $(b_M - 2) \frac{\log s_M}{\log r} > N$ for $M$ of the order of $(\log N)^2$.

Naively finding the minimum of $A'_m$ in each step $m\leq M$ by calculating all values $A'_m(x)$ for $x$ in the set $\sigma_m(\xi_{m-1})$ costs $O(e^{m^{1/2}})$ computations of a complex number of the form $e(r_i^j tx)$ for each of the $2^{b_m-a_m -2} = O(2^{e^{m^{1/2}}})$ elements $x$ in $\sigma_m(\xi_{m-1})$. Hence in each step $m$ we need to perform
$e^{m^{1/2}} \cdot 2^{e^{m^{1/2}}} = O(N2^N)$ 
mathematical operations. Carrying out $M \approx (\log N)^2$ many steps of the algorithm, this are in total
\begin{equation*}
O(N 2^N (\log N)^2) = O(e^N)
\end{equation*}
many mathematical operations.

\subsection{Discrepancy}
We fix a base $r \geq 2$ and $t$ shall denote a non-zero integer. For a large natural number $N$, using Schmidt's Hilfssatz 7, the Erd\H{o}s-Tur\'{a}n inequality, and via approximating $N$ by a suitable value $\langle M; r \rangle$ we can find an upper bound for the discrepancy $D_N(\{r^n \xi\})$.

\begin{Thm}
The discrepancy of Schmidt's absolutely normal number $\xi$ is
\begin{equation}
D_N ( \{r^n \xi \} ) \ll \frac{\log \log N}{\log N}
\end{equation}
where the implied constant and `$N$ large enough' depend on the base $r$.
\end{Thm}

\begin{proof}
For a given $N$ large enough, let $M$ such that $\langle M ; r \rangle \leq N < \langle M+1 ; r \rangle$. Such an $M$ is of size $O((\log (\log r \cdot N))^2)$.

We split the Weyl sum $\sum_{n=1}^N e(r^n t \xi)$ according to
\begin{equation}\label{Weyl_sum_schmidt}
\sum_{n=1}^N e(r^n \xi t) = \sum_{n=1}^{\langle M ; r \rangle } e(r^n \xi t) + \sum_{n=\langle M ; r \rangle + 1}^N e(r^n \xi t).
\end{equation}

An estimate for the first sum $\sum_{n=1}^{\langle M; r \rangle} e(r^n \xi t)$ in \ref{Weyl_sum_schmidt} can be obtained from \ref{A_m'_cancellation_lemma} and yields
\begin{align*}
\sum_{n=1}^{\langle M; r \rangle} e(r^n t \xi)
&= \sum_{m=m_0(r)}^{M-1} \sum_{n=\langle m; r \rangle+1}^{\langle m+1; r \rangle} e(r^n t \xi) + O(1) \\
& \ll \sum_{m=m_0(r)}^{M-1} m ( \langle m+1 \rangle - \langle m \rangle )^{1-\frac{\beta_m}{2}} \\
& \leq M \sum_{m=1}^{M-1} ( \langle m+1 \rangle - \langle m \rangle)^{1-\frac{\beta_{M}}{2}} \\
&< M^2 \left( \sum_{m=1}^{M-1} \langle m+1 \rangle - \langle m \rangle \right)^{1- \frac{\beta_{M}}{2}}
\end{align*}
which is equal to $M^2 \langle M \rangle^{1-\frac{b_{M}}{2}}$. 
Using the decay property $\beta_M \geq \beta_1 M^{-1/2}$, the first sum in \ref{Weyl_sum_schmidt} is thus
\begin{align}\label{Estimate for Weyl sums}
\ll M^2 e^{M^{1/2} - \frac{\beta_1}{2} M^{1/4}}.
\end{align}

For the error of approximation of $N$ via $\langle M ; r \rangle$ we calculate for fixed $r,s_1$, and $M$ large enough,
\begin{equation}\label{symbol_m_diff}
\langle M+1; r \rangle - \langle M; r \rangle \ll e^{\sqrt{M}} \left( e^{\frac{1}{2 \sqrt{M}}}-1 + \frac{M^2}{e^{\sqrt{M}}} \right)
\end{equation}
where the implied constant depends on $s_1$ and $r$. We used $\sqrt{M+1} - \sqrt{M} = 1/(\sqrt{M+1}+\sqrt{M}) \leq 1/2\sqrt{M}$. For $M$ large enough we have $e^{1/2\sqrt{M}} \leq 1+ \frac{1}{\sqrt{M}}$, hence \ref{symbol_m_diff} is
\begin{align*}
&\leq e^{\sqrt{M}} \left( \frac{1}{{\sqrt{M}}} + \frac{M^2}{e^{\sqrt{M}}} \right).
\end{align*}
Thus,
\begin{equation}\label{ErrorNviabm}
\langle M+1; r \rangle - \langle M; r \rangle = e^{\sqrt{M}} \cdot O\left( \frac{1}{\sqrt{M}} \right)
= \langle M; r \rangle \cdot O\left( \frac{1}{\sqrt{M}} \right).
\end{equation}
With $M$ of order $(\log(\log r \cdot N))^2$, this is
\begin{equation*}
\ll \frac{N}{\log (\log r \cdot N)} \sim \frac{N}{\log N},
\end{equation*}
hence the second term in \ref{Weyl_sum_schmidt} dominates the first.

The Erd\H{o}s-Tur\'{a}n inequality applied to the sequence $\{r^n \xi \}_{n\geq 0}$ is
\begin{equation}\label{ErdosTuran}
D_N(\{r^n \xi \}) \ll \frac{1}{H} + \sum_{t=1}^H \frac{1}{t} \left\vert \frac{1}{N} \sum_{n=1}^N e(r^n \xi t) \right\vert
\end{equation}
where $H$ is a natural number. Splitting the exponential sum as before and upon putting $H = \log N$, we thus obtain
\begin{equation*}
D_N(\{r^n \xi \}) \ll  \frac{\log \log N}{\log N}
\end{equation*}
where the implied constant depends on the base $r$.
\end{proof} 

\subsection{Modifying Schmidt's Algorithm.}\label{Modifying_Schmidt}

We show that it is possible to modify Schmidt's algorithm for a given $A>0$ to output an absolutely normal number $\xi$, depending on $A$, with discrepancy $D_N(\{r^n \xi \}) = O_r( \frac{\log \log N}{(\log N)^A})$ to base $r$, where the implied constant depends on $r$. This convergence is simultaneously faster than the speed of convergence to normality of most constructions of normal numbers (to a single base) by concatenations of blocks (see for example \cite{drmota_tichy} and \cite{schiffer1986:discrepancy_of_normal_numbers}).

\begin{Prop}
Fix $0<c<1$. Schmidt's algorithm still holds when the symbol $\langle m \rangle$ is replaced by the function
\begin{equation*}
\langle m \rangle = \lfloor e^{m^c} \rfloor.
\end{equation*}
\end{Prop}

Note that the symbols $\langle m;r \rangle$, $a_m$ and $b_m$ and also the construction of the sets $\sigma_m$ have to be modified accordingly. The algorithm works in exact the same way, but the output will depend on $c$.

\begin{proof}
We need to show that the estimate \ref{A_m'estimate} for $A_m'$ is still valid with this choice of $\langle m \rangle$. In course of the proof of this estimate, Schmidt evaluates the inner sum over $j$ in $A_m'$ trivially on a range of size $O(m)$. This range constitutes only a minor part of the full sum over $j$ since $m \leq \delta (\langle m+1 \rangle - \langle m \rangle)^{1-\epsilon}$ for some $\delta>0$ and some $0<\epsilon<1$. This can be seen from
\begin{align*}
e^{(m+1)^c} - e^{m^c} &= e^{m^c} \left(e^{(m+1)^c - m^c} - 1 \right) \\
&\geq e^{m^c} \left( (m+1)^c - m^c \right) \\
&\gg c m^{c\alpha} m^{c-1}
\end{align*}
since $e^{m^c} \gg m^{c \alpha}$ for any $\alpha >0$. Choosing $\alpha  = \frac{2}{c}-1 + \eta$ for some $\eta>0$ gives
\begin{equation*}
\langle m+1 \rangle - \langle m \rangle \gg c m^{c\alpha} m^{c-1}  = c m^{1+\eta}
\end{equation*}
which establishes our claim.
\end{proof}

The  discrepancy of $\xi = \xi_c$ can be estimated the same way as before. Note that any $N$ large enough can now be approximated by the symbol $\langle M \rangle$ with error
\begin{equation*}
O(\langle M+1 \rangle - \langle M \rangle) \ll e^{M^{c}}\frac{1}{M^{1-c}}
\end{equation*}
which with  $M$ of order  $(\log N)^{1/c}$ is
\begin{equation*}
\frac{N}{(\log N)^{\frac{1-c}{c}}}.
\end{equation*}
Hence the discrepancy of the sequence $\{r^n \xi\}_{n\geq 0}$ satisfies
\begin{equation}
D_N(\{r^n \xi\}) \ll \frac{\log \log N}{(\log N)^A}
\end{equation}
with $0<A=\frac{1-c}{c}<\infty$.

\section{The Constants $a_{20}$ in Schmidt's Hilfssatz 5}\label{section:constants}

In this section we will prove the following explicit variant of Schmidt's Hilfssatz 5 in \cite{schmidt1961:uber_die_normalitat}.

\begin{Lem}[Explicit variant of Lemma \ref{Schmidt:HS5}]\label{HS5_explicit}
Let $r$ and $s$ be integers greater than $1$ such that $r\not\sim s$. Let $K, l$ be positive integers such that $l \geq s^K$ and denote $m=\max(r,s)$. Then for 
\begin{equation}
N \geq N_0(r,s) = \exp(288\cdot(12 m (\log m)^4 + 8 (\log m)^3 + (\log m)^2))
\end{equation}
we have
\begin{equation}
\sum_{n=0}^{N-1} \prod_{k=K+1}^\infty \vert \cos(\pi r^n l /s^k)\vert \leq 2N^{1-a_{20}} \label{HS5:sum_explicit_variant}
\end{equation}
for some positive constant $a_{20}$ as specified in \ref{a_20} that satisfies
\begin{equation}
a_{20} \approx 0.0057 \cdot \frac{1}{s^4 \log s} \left( \frac{1}{\log s} - \frac{1}{s}\right).
\end{equation}
\end{Lem}

\begin{Rem}
The statement of Lemma \ref{HS5_explicit} holds true \emph{for all} $N$ with
\begin{equation}\label{a_20_trivialized}
a_{20} = \min \left( \frac{1}{N_0 \log N_0}, \frac{-\log \cos (\frac{\pi}{s^2})}{2\log N_0} \right)
\end{equation}
as specified in \ref{a_20_all_N}.
\end{Rem}

This enables us in principle to give an explicit description of the sequences $(r_i)_{i\geq 1}$ and $(s_j)_{j\geq 1}$ \emph{after} the repetition of the entries via the function $\phi$ as suggested by Schmidt. Lemma \ref{HS5_explicit} might also be of independent interest as its non-explicit variant has been used by several authors, see e.g. \cite{becher_bugeaud_slaman2013:simply_normal} and \cite{Becher_slaman2014:on_the_normality_of_numbers_to_different_bases}.
We do not claim optimality of the bounds in Lemma \ref{HS5_explicit}.

\begin{proof}
The proof is basically a careful line-by-line checking of Schmidt's proof of Lemma \ref{Schmidt:HS5}. The reader might find it helpful having a copy both of \cite{schmidt1960:on_normal_numbers} and \cite{schmidt1961:uber_die_normalitat} at hand.

We follow Schmidt's notation and his argument in \cite{schmidt1961:uber_die_normalitat}.

Let
\begin{align*}
r &= p_1^{d_1} \cdot \ldots \cdot p_h^{d_h}, \\
s &= p_1^{e_1} \cdot \ldots \cdot p_h^{e_h}
\end{align*}
be the prime factorizations of $r$ and $s$ with $d_i$ and $e_i$ not both equal to zero. We assume the $p_i$ to be ordered such that
\begin{equation*}
\frac{d_1}{e_1} \geq \ldots \geq \frac{d_h}{e_h},
\end{equation*}
with the convention that $\frac{d}{0} = + \infty$. This implies that $d_k e_l - d_l e_k \geq 0$ for all $k \geq l$.

Let $b = \max_i(d_i ) \cdot \max_i(e_i)$. Schmidt denotes by $l_i$ numbers not divisible by $p_i^{2b}$.

For $1 \leq i \leq h$, let
\begin{align*}
u_i &= (p_1^{d_1} \ldots p_i^{d_i})^{e_i} (p_1^{e_i}\ldots p_i^{e_i})^{-d_i} \\ 
v_i &= (p_{i+1}^{e_{i+1}} \ldots p_h^{e_h})^{d_i} (p_{i+1}^{d_{i+1}} \ldots p_h^{d_h})^{-e_i},
\end{align*}
where empty products (for $i=h$) are $1$. These numbers are integers, and $t_i = \frac{u_i}{v_i}$ is not equal to $1$ since $r\not\sim s$. We have $t_i = \frac{r^{e_i}}{s^{d_i}}$, hence, when writing $t_i$ in lowest terms, the prime $p_i$ has been cancelled.

Let $f_i = p_i-1$ if $p_i$ is odd, and $f_i = 2$ otherwise. There are well-defined integers $g_i$ such that
\begin{equation*}
t_i^{f_i} \equiv 1 + q_i p_i^{g_i - 1} \pmod{p_i^{g_i}}
\end{equation*}
with $p_i \ndiv q_i$ (especially $q_i \neq 0$). We have $g_i > 1$ by the small Fermat theorem and for $p_i=2$ we even have $g_i>2$ since squares are congruent $1$ modulo $4$.
 To give an upper bound for $g_i$, note that  $p_i^{g_i}$ can be at most equal to $t_i^{f_i}$. 
 Hence $g_i \leq \lfloor f_i \frac{\log t_i}{\log p_i} \rfloor + 1$.
  Since naively $\log p_i \geq \log 2$, $\log t_i = e_i \log r - d_i \log s \leq e_i \log r \leq \frac{\log r \log s}{\log 2}$ and $p_i \leq \max(r,s)$, a trivial upper bound on $g_i$, valid for all $i$, is
\begin{equation*}
g_i \leq 12 \max(r,s) \log r \log s.
\end{equation*}

Let $a_1 = \max (g_1, \ldots, g_h)$. Then
\begin{equation}
2  \leq a_1 \leq 12 \max(r,s) \log r \log s.
\end{equation}

Assume $k \geq a_1$, $e_i > 0$. The constant $a_2$ is such that at most $a_2 (s/2)^k$ of the numbers $l_i t_i^n$ fall in the same residue class modulo $s^k$ if $n$ runs through a set of representatives modulo $s^k$ (Hilfssatz 1 in \cite{schmidt1961:uber_die_normalitat}). At most $p_i^{2b} p_i^{g_i}$ of the numbers $t_i^n l_i$ fall in the same residue class modulo $p_i^k$ if $n$ runs through a set of representatives modulo $p_i^k$. If $p_i | s$, then there are at most $(s/2)^k$ elements in a set of representatives modulo $s^k$ that are congruent to each other. Hence $a_2 = \max_{i, e_i>0} p_i^{2b + g_i}$. Naive upper and lower bounds on $a_2$ are thus
\begin{equation}
8 \leq a_2 \leq \max(r,s)^{8 \log(\max(r,s)) + 12 \max(r,s) \log r \log s}.
\end{equation}

The constant $a_4$ (named $\alpha_3$ in \cite{schmidt1960:on_normal_numbers}) is chosen  such that
\begin{equation}\label{a_4}
(s^2-2)^{a_4} < 2^{1/4 + 2a_4} a_4^{a_4} (1-2a_4)^{1/2 - a_4}.
\end{equation}
The right-hand side of \ref{a_4} as a function of $a_4$ (denote it by $f(a_4)$) can be numerically analyzed. It is a strictly decreasing continuous function on the interval $(0, 1/16]$ with values $f(0^+) = \sqrt[4]{2} \approx 1.19 > f(1/16) \approx 1.028 >1$. Hence any $a_4$ in $(0,1/16)$ that satisfies
\begin{equation}\label{a_4modified}
(s^2-2)^{a_4} \leq f(1/16)
\end{equation}
also satisfies \ref{a_4}. Note that $a_4 < 1/16$ is no proper restriction as $a_4(2) \approx 0.055 < 1/16$ and since $a_4$ is decreasing in $s$. Now, \ref{a_4modified} is easy to solve and gives
\begin{equation*}
a_4(s) \geq \frac{c}{\log(s^2-2)}
\end{equation*}
for $c= \log(f(1/16)) \approx 0.028$. This constitutes a non-trivial (i.e. positive) lower bound on the values of $a_4$ that are admissible. To simplify matters we continue with this value for $a_4$, i.e. we put
\begin{equation}
a_4 = \frac{0.028}{\log(s^2-2)}.
\end{equation}

The constant $a_3$ also comes from the earlier Schmidt paper \cite{schmidt1960:on_normal_numbers} and was called $\alpha_4$ there. Schmidt counts the number blocks of digits in base $s$ with few `nice' digit pairs. These are successive digits not both equal to zero or $s-1$. He derives the proof of Lemma 3 in \cite{schmidt1960:on_normal_numbers} that the number of combinations of $k$ base $s$ digits with less than $\alpha_3 k (= a_4 k)$ nice digit pairs, counting only non-overlapping pairs of digits, does not exceed
\begin{equation}\label{number_of_digit_pairs_odd_indices}
k \binom{\frac{k}{2}}{\lfloor a_4 k \rfloor} (s^2 - 2)^{\lfloor a_4 k \rfloor} 2^{k/2 - \lfloor a_4 k \rfloor} .
\end{equation}
With the approximation
\begin{equation*}
\sqrt{2 \pi} n^{n+1/2} e^{-n} \leq n! \leq e n^{n+1/2} e^{-n}
\end{equation*}
we find that \ref{number_of_digit_pairs_odd_indices} is
\begin{align}
& \leq k  \frac{e}{2 \pi}  
\frac{\left( \frac{k}{2}\right)^{k/2 + 1/2}   }{(a_4k)^{a_4k + 1/2} \left(\frac{k}{2}-a_4 k \right)^{k/2-a_4k+1/2} }
(s^2-2)^{a_4k} 2^{k/2 - a_4 k} 2^{-1}\nonumber \\
&=  \frac{e}{2 \pi}
\frac{1}{\sqrt{a_4} \sqrt{1-2a_4}} \frac{1}{2} \cdot
\frac{1}{(2a_4)^{a_4k} (1-2a_4)^{(1/2-a_4)k}}. \label{Schmidt_lemma_3_ineq}
\end{align}
In \cite{schmidt1960:on_normal_numbers}, Schmidt denotes the constant factor by $\alpha_5$,
\begin{equation*}
\alpha_5 = \frac{e}{4 \pi \sqrt{a_4(1-2a_4)}}.
\end{equation*}
Using $a_4 \leq 0.055$ and $a_4 \geq \frac{c}{\log(s^2-2)}$ with $c \approx 0.028$ we obtain the upper bound
\begin{equation}
\alpha_5 \leq 1.87 \sqrt{\log s} \approx 2 \sqrt{\log s}.
\end{equation}
Finally, $a_3$ is such that if $k \geq a_3$, and respecting the choice of $a_4$, then
\begin{equation*}
\alpha_5 k \frac{(s^2-2)^{a_4k} 2^{(1/2-a_4)k}}{(2a_4)^{a_4k}(1-2a_4)^{(1/2-a_4)k}} < 2^{3/4 \cdot k}
\end{equation*}
holds. The left-hand side is equal to
\begin{equation*}
\alpha_5 k  \left( \frac{ (s^2-2)^{a_4}}{f(a_4)} 2^{3/4} \right)^k 
= \alpha_5 k \left( \frac{f(1/16)}{f(a_4)} 2^{3/4} \right)^k
\leq \alpha_5 k 2^{0.74 k}
\end{equation*}
by the choice of $a_4$ and since $f(a_4) \geq f(0.055) > f(1/16)$. Using $\log(x) \leq x^{1/2}$ for all $x\geq0$,
\begin{equation*}
\alpha_5 k 2^{0.74 k} < 2^{3/4\cdot k}
\end{equation*}
is satisfied for all $k$ larger than
\begin{equation}
a_3 = 120 \sqrt{\log(s)}.
\end{equation}

Let $N \geq \max(s^{a_1}, s^{a_3+1})$ (hence certainly $N \geq s^2$ since $a_1 \geq 2$) and let $k$ be such that $s^k \leq N < s^{k+1}$. The constants $a_7$ and $a_5$ in Hilfssatz 2 in \cite{schmidt1961:uber_die_normalitat} are such that
\begin{equation*}
a_2 (s/2)^k s 2^{3k/4} < a_7 N^{1-a_5}.
\end{equation*}
With $k > \frac{\log N}{\log s}-1$ the left-hand side is
\begin{equation*}
< a_2 s N^{1-\frac{\log 2}{4}\left(\frac{1}{\log s}-\frac{1}{\log N}\right)}
\end{equation*}
which is
\begin{equation}\label{a_7}
\leq a_2 s N^{1-\frac{\log 2}{8 \log s}}
\end{equation}
due to $N\geq s^2$. Hence $a_7 = a_2 s$. We want \ref{a_7} to be $< N^{1-\frac{\log 2}{16 \log s}}$, hence 
\begin{equation}
a_5 = \frac{\log 2}{16 \log s} > 0. 
\end{equation}
This happens, if 
\begin{equation*}
\frac{\log(a_2 s)}{\log N} < \frac{\log 2}{16 \log s}
\end{equation*}
which is satisfied when $N \geq N_0^{\text{HS2}}$, where
\begin{equation}\label{N_large}
\log N_0^{\text{HS2}} = 288m (\log m)^4 + 192 (\log m)^3 + 24 (\log m)^2
\end{equation}
where we denoted $m = \max(r,s)$. Note that in particular $N$ is much larger than $e^s$. 

The constant $a_6$ is such that $a_4 k > a_6 \log N$. With $a_4 \geq \frac{0.028}{\log(s^2-2)} > \frac{0.014}{\log s}$ and $k > \frac{\log N}{\log s} -1$ and due to $N \geq e^s$, we have
\begin{equation*}
a_4 k > \log N \frac{0.014}{\log s} \left( \frac{1}{\log s} - \frac{1}{s} \right).
\end{equation*}
This is a positive value for all $s$. Hence
\begin{equation}
a_6 = \frac{0.014}{\log s} \left( \frac{1}{\log s} - \frac{1}{s}\right) >0
\end{equation}
is an admissible choice for $a_6$.

In Hilfssatz 3 in \cite{schmidt1961:uber_die_normalitat}, Schmidt divides the numbers $lr^n$ in at most $hb$ sequences each of which having a certain number of elements. If this number is $\leq N^{1/2}$, he counts trivially. If the number of elements in a sequence is larger than $N^{1/2}$, he uses Hilfssatz 2 with this $N$. Hence the $N$ in Hilfssatz 3 needs to be large enough such that $N^{1/2}$ is large enough for Hilfssatz 2. Thus, for
\begin{equation}
N \geq N_0^{\text{HS3}} = (N_0^{\text{HS2}})^2 = \exp(2\cdot(288m (\log m)^4 + 192 (\log m)^3 + 24 (\log m)^2))
\end{equation}
there are at most $hbN^{1-a_5}$ numbers of the $lr^n$ having less than $a_6 \log \sqrt{N}$ nice digit pairs. Hence
\begin{equation}
a_{9} = \frac{a_6}{2} = \frac{0.007}{\log s}\left( \frac{1}{\log s}- \frac{1}{s}\right).
\end{equation}
The number $h$ was defined as the number of distinct prime divisors in $rs$, for which a trivial bound is $h\leq \log_2(rs)\leq \frac{2 \log m}{\log 2}$ with $m = \max(r,s)$. Another trivial bound is $b = \max(d_i) \cdot \max(e_i) \leq (\log_2(m))^2$. Thus with $N \geq e^{288m}$, we have
\begin{equation*}
hb N^{1-a_5} \leq 7 (\log m)^3 N^{1-a_5} \leq  N^{1 - a_8}
\end{equation*}
with
\begin{equation}
a_8 = a_5 - \frac{\log 7 + 3 \log \log m}{288 m} = \frac{\log 2}{16 \log s} - \frac{\log 7 + 3 \log \log m}{288 m}.
\end{equation}

Recall from Schmidt's paper that $z_K(x)$ denotes the number of nice digit pairs $c_{i+1} c_i$ of $x$ with $i\geq K$ where the $c$ are the digits of $x$ in base $s$.

In Hilfssatz 4, Schmidt begins with the restriction $ n \geq N^{2/3} \log s/ \log r$ which reduces $a_{14}$ to a value less than $1/3$. The remaining numbers $lr^n$ are divided in at most $2N^{2/3}$ many intervals of length $\lfloor N^{1/3} \rfloor$ which are analyzed separately. 
The restriction $n \geq N^{2/3} \log s/ \log r$ implies $lr^n \geq s^{K + \lfloor N^{1/3} \rfloor^2}$.
Schmidt  wants to apply Hilfssatz 3 to intervals $N^{2/3} \log s/ \log r \leq n_0 \leq n < n_0 + \lfloor N^{1/3} \rfloor$ of length $\lfloor N^{1/3} \rfloor$. However, he makes one further preliminary reduction in showing that one can assume that $z_K(l)$ is less than $\frac{a_9}{2} \log N$. He denotes by $n_1$ the least $n$ such that $z_K(lr^n) <\frac{a_9}{2} \log N$ and replaces $lr^n$ for $n\geq n_1$ by $l^\ast r^{n-n_1}$ where $l^\ast = lr^{n_1}$. All $lr^n$ with $ n_0 \leq n < n_1$ are by the choice of $n_1$ such that $z_K(lr^n) \geq \frac{a_9}{2} \log N$. As Schmidt's version is not explicit, he can assume $N$ to be large enough, and apply Hilfssatz 3 to the interval $ n_1 \leq n < N$ (or $0 \leq n < N-n_1$ for numbers $l^\ast r^n$). To make things explicit, we distinguish three cases for the size of $n_1$. We write $M = \lfloor N^{1/3} \rfloor$ for the number of $lr^n$ under consideration. We want to find explicit lower bounds on $M$ such that we can apply Hilfssatz 3.

Case 0: $n_1$ does not exist at all. Then the number of $lr^n$ with $z_K$ less than $a_9 \log M$ is trivially less than $M^{1-a}$ for any $0<a<1$.

Case 1: $n_1$ is large such that the number of $lr^n$ with $z_K < a_9 \log M$ can be trivially estimated by $M-n_1 \leq M^{1-a_8}$. This is the case when $n_1 \geq M - M^{1-a_8}$.

Case 2: $n_1 < M - M^{1-a_8}$. We need the interval $M-n_1$ to be large enough to be able to apply Hilfssatz 3 to obtain cancellation, i.e. $M-n_1 \geq N_0^{\text{HS3}}$ which holds if $M \geq M_0 =  (N_0^{\text{HS3}})^{\frac{1}{1-a_8}}$. Thus by Hilfssatz 3 the number of $lr^n$, $n_0 \leq n < n_0 + M$, with $z_K < a_9 \log N$ is at most $(M - n_1)^{1-a_8} \leq M^{1-a_8}$.

Schmidt uses a reduction to count only $z_K$ instead of all nice digit pairs. This reduction looses at one point $2$ digit pairs, i.e. after an application of Hilfssatz 3 one finds numbers with at most $a_9 \log M - 2$ nice digit pairs. This is $\leq \frac{a_9}{2} \log M$ for
\begin{equation}\label{rand1}
\log M > \frac{4}{a_9}.
\end{equation}
Also, Schmidt's reduction works if
\begin{equation}\label{rand2}
M \frac{\log r}{\log s} < \left\lfloor \frac{M^2-1}{\frac{a_9}{2} \log M }\right\rfloor - 1.
\end{equation}
Note that \ref{rand1} and \ref{rand2} do not pose further restrictions on $M$.

Finally, from $M \geq (N_0^{\text{HS3}})^{\frac{1}{1-a_8}}$, $M = \lfloor N^{1/3} \rfloor$, and since we may assume that $a_8<\frac{1}{2}$, the requirement 
\begin{equation}\label{N_0HS4}
N \geq N_0^{\text{HS4}} = (N_0^{\text{HS3}})^{6} = \exp(288\cdot(12 m (\log m)^4 + 8 (\log m)^3 + (\log m)^2))
\end{equation}
for the original $N$ follows. We established that in each subsequence of length $\lfloor N^{1/3} \rfloor$ there are at most $\lfloor N^{1/3} \rfloor ^{1-a_8}$ elements $lr^n$ with $z_K < \frac{a_9}{2} \log \lfloor N^{1/3} \rfloor$.

In total, since there are at most $2N^{2/3}$ many intervals for $n$ of length $\lfloor N^{1/3} \rfloor$, we obtain (for $\log N \geq \frac{6 \log 3}{a_8}$) that there are at most
\begin{equation*}
N^{2/3} \frac{\log s}{\log r} + 2N^{2/3} \cdot \lfloor N^{1/3} \rfloor^{1-a_8} \leq 3 N^{1-a_8/3} \leq N^{1-a_8/6}
\end{equation*}
elements $r^n l$, $0\leq n < N$, with $z_K < \frac{a_9}{2}\log \lfloor N \rfloor^{1/3} \leq \frac{a_9}{6} \log N$. Thus
\begin{equation}
a_{14} = \frac{a_8}{6}, \quad a_{15} = \frac{a_9}{6}.
\end{equation}

From Hilfssatz 5 follows that $a_{20} = \min(a_{14}, a_{22})$ where $ a_{22} =  - a_{15} \log a_{21}$ with $a_{21} = \cos(\pi / s^2)$. We have $-\log a_{21} =- \log \cos(\frac{\pi}{s^2}) \geq \frac{\pi^2}{2} \frac{1}{s^4}$. Plugging in the values of $a_{14}$ and $a_{15}$ shows that $\min(a_{14}, a_{22}) = a_{22}$. Hence
\begin{equation}\label{a_20}
a_{20} = \frac{1}{6} \frac{\pi^2}{2} \cdot 0.007 \cdot \frac{1}{s^4} \frac{1}{\log s} \left( \frac{1}{\log s} - \frac{1}{s}\right)
\end{equation}
where the constant factor is approximately $0.0057$.

To find $a_{20}$ such that Lemma \ref{HS5_explicit} holds \emph{for all} $N$, we need to replace $a_{14}$ and $a_{15}$ by sufficiently small constants such that Hilfssatz 4 holds for all $N$. This can be achieved by 
\begin{equation*}
a_{14} \leftarrow \min(a_{14}, 1-\frac{\log(N_0-1)}{\log N_0}), \quad \text{and} \quad 
a_{15} \leftarrow \min(a_{15}, \frac{1}{2 \log N_0})
\end{equation*}
where we denoted $N_0 = N_0^{\text{HS4}}$. We have $a_{14}^{\text{old}} \approx 0.007 \frac{1}{\log s}$ and $1-\frac{\log(N_0-1)}{\log N_0} \leq \frac{2}{N_0 \log N_0}$ which decays worse than exponentially in $m$. Furthermore, $a_{15}^{\text{old}} \approx 0.001 \frac{1}{\log s}$ and $\frac{1}{2 \log N_0}$ is worse than linear with in $m$ a large constant. Note also $1-\frac{\log(N_0-1)}{\log N_0} \geq \frac{1}{N_0 \log N_0}$. Hence Hilfssatz 4 holds true for all $N$ with constants
\begin{equation}\label{a_14_and_a_15_for_all_N}
a_{14} = \frac{1}{N_0 \log N_0}, \quad \text{and} \quad a_{15} = \frac{1}{2\log N_0}
\end{equation}
with $N_0 = N_0^{\text{HS4}}$ as in \ref{N_0HS4}.

The constant $a_{20}$ then modifies according to
\begin{equation}\label{a_20_all_N}
a_{20} = \min(a_{14}, a_{22}) =  \min \left( \frac{1}{N_0 \log N_0}, \frac{-\log \cos (\frac{\pi}{s^2})}{2\log N_0} \right).
\end{equation}
For large $m$, $a_{20}$ will equal $a_{14}$ but since $a_{22} \geq \frac{\pi^2}{4} \frac{1}{s^4 \log N_0}$, for small $m$ we have $a_{20} = a_{22}$. Explicitly, with $a_{14} \leq \frac{1}{e^m 1728 m}$ we have
\begin{equation}
a_{20} = \frac{1}{N_0 \log N_0}
\end{equation}
for $m \geq 7$ were we denoted $m=\max(r,s)$ and $N_0 = N_0^{\text{HS4}}$ as given in \ref{N_0HS4}.
\end{proof}

\appendix
\section{Algorithms by Sierpinski and Turing}
\subsection{Sierpinski's Algorithm}\label{Sierpinksi_section}

In this section we will estimate the runtime and discrepancy of the effective version of Sierpinsk's algorithm \cite{sierpinski1917:borel_elementaire} by Becher and Figueira \cite{becher_figueira:2002}. This algorithm is double exponential but with discrepancy $O(\frac{1}{N^{1/6}})$.

Let $0 < \epsilon \leq \frac{1}{2}$ be a rational (or computable real) number that will be fixed throughout the algorithm. We also choose in advance a base $b\geq 2$. The algorithm will compute the digits to base $b$ of an absolutely normal number $\nu$. The output (i.e. $\nu$) depends on the choice of $\epsilon$ and $b$.

\subsubsection*{Notation}
Let $m$, $q$, $p$ be integers such that $m\geq 1$, $q\geq 2$ and $0\leq p \leq q-1$ and put $n_{m,q} = \lfloor \frac{24m^6q^2}{\epsilon}\rfloor + 2$.

Let $\Delta_{q,m,n,p}$ be the interval $(\frac{0.b_1 \ldots b_{n-1}(b_n-1)}{q^n}, \frac{0.b_1 \ldots b_{n-1}(b_n+2)}{q^n})$ where the string $b_1\ldots b_n$ is such that the digit $p$ appears too often, i.e. $\vert \frac{N_p(b_1 \ldots b_n)}{n}- \frac{1}{q}\vert \geq \frac{1}{m}$ where $N_p(b_1 \ldots b_n)$ denotes the number of occurrences of the digit $p$ amongst the $b_i$.

Let
\begin{equation*}
\Delta = \bigcup_{q=2}^\infty \bigcup_{m=1}^\infty \bigcup_{n=n_{m,q}}^\infty \bigcup_{p=0}^{q-1} \Delta_{q,m,n,p},
\end{equation*}
and denote a truncated version of $\Delta$ by
\begin{equation*}
\Delta_k = \bigcup_{q=2}^{k+1} \bigcup_{m=1}^k \bigcup_{n=n_{m,q}}^{k\cdot n_{m,q}} \bigcup_{p=0}^{q-1} \Delta_{q,m,n,p}.
\end{equation*}
The complement of $\Delta$ in $[0,1)$ is
\begin{equation*}
E = [0,1) \setminus \Delta.
\end{equation*}

Sierpinski's algorithm will compute the digits to base $b$ of a number $E$. This number is absolutely normal as shown by Sierpinski and in Theorem 7 in \cite{becher_figueira:2002}.

The truncated sets $\Delta_k$ approximate $\Delta$ in the sense that if a number does not lie in $\Delta_k$ for large enough $k$, then it will also not lie in $\Delta$. Becher and Figueira's algorithm computes the digits of $\nu$ such that the $n$-th digit ensures that $\nu$ does not lie in some $\Delta_{p_n}$, where $p_n \rightarrow \infty$.

\subsubsection*{The Algorithm}

First digit: Split the unit interval in subintervals $c^1_d = [\frac{d}{b}, \frac{d+1}{b})$ for $0\leq d < b$. Put $p_1 = 5 \cdot (b-1)$. Compute the Lebesgue measure of $\Delta_{p_1} \cap c_d^1$ for all $d$. The first digit $b_1$ of $\nu$ is chosen such that it is (the smallest among) the $d$ such that the Lebesgue measure of $\Delta_{p_1} \cap c_{b_1}^1$ is minimal among the $\Delta_{p_1} \cap c_d^1$.
\medskip

$n$-th digit: Split the interval $[0.b_1 \ldots b_{n-1}, 0. b_1 \ldots (b_{n-1} +1) )$ in subintervals $$c_d^n = [0.b_1 \ldots b_{n-1} d, 0.b_1 \ldots b_{n-1} (d+1))$$ for all $0\leq d < b$. Put $p_n = 5 \cdot (b-1) \cdot 2^{2n-2}$. 
The $n$-th digit $b_n$ of $\nu$ will be the (smallest of the) $d$ that minimize the Lebesgue measure of the $\Delta_{p_n} \cap c_d^n$.

\subsection{Runtime}

For fixed $q$, $m$, $n$ and $p$, writing down all strings $b_1 \ldots b_n$ of length $n$ of digits $0\leq b_i < b$ that satisfy the conditions of $\Delta_{q,m,n,p}$ takes exponential time in $n$. Naively estimating gives the complexity of computing $\Delta_k$ as being exponential in $k$. Since in step $n$ we have chosen $k=p_n$ which is exponential in $n$, the complexity of computing $\Delta_{p_n}$ takes double exponential elementary operations.

\subsection{Discrepancy}

We will give an estimate for the discrepancy of a generic element of the set $E$, not taking into account that the algorithm might in fact construct an element with better distributional properties.

The family of intervals $\bigcup_{p=0}^{q-1} \Delta_{b,m,n,p}$ contains all real numbers with expansion to base $b$ not simply normal regarding the first $n$ digits. The union
\begin{equation*}
\Delta_{q,m} = \bigcup_{n=n_{m,q}}^\infty \bigcup_{p=0}^{q-1} \Delta_{q,m,n,p}
\end{equation*}
contains all real numbers whose base-$q$ expansion is not simply normal for any large enough number of digits considered. Hence any $\nu$ not in $\Delta_{q,m}$ satisfies
\begin{equation*}
\left\vert \frac{\sharp \{ \{q^n \nu \} \in I \mid n\leq N \}}{N} - \vert  I \vert \right\vert < \frac{1}{m}
\end{equation*}
for all $N\geq n_{m,q}$ and $I$ of the form $I = [\frac{p}{q}, \frac{p+1}{q})$ , $p=0, \ldots, q-1$. 

Inverting the relation between $N$ and $m$ and using Sierpinski's choice for $n_{m,q} = \lfloor \frac{24 m^6 q^2}{\epsilon} \rfloor + 2 \sim \frac{24 m^6 q^2}{\epsilon}$, we find that
\begin{equation*}
\sup_{p=0, \ldots q-1} \left\vert \frac{\sharp \{ \{q^n \nu \} \in [\frac{p}{q}, \frac{p+1}{q}) \mid n\leq N \}}{N} - \vert  I \vert \right\vert \leq \left( \frac{24}{\epsilon} \right)^{1/6} \frac{q^{1/3}}{N^{1/6}} + O\left(\frac{1}{N^{1/3}}\right) \ll_\epsilon q^{1/3} \frac{1}{N^{1/6}}
\end{equation*}
where the implied constant depends on $\epsilon$ but not on $q$.

Fix $I\subset [0,1)$, $\delta > 0$ and $k$ such that $\frac{2}{q^k} < \delta$. Choose $l,m$ such that $I \subset [\frac{l}{q^k}, \frac{m}{q^k})$ and $\vert I \vert < \frac{m-l}{q^k} + \frac{2}{q^k}$. Then
\begin{align*}
\frac{\sharp \{ \{q^n \nu\} \in I \mid n\leq N \}}{N}
& \leq \frac{\sharp \{ \{q^n \nu\} \in [\frac{l}{q^k}, \frac{m}{q^k}) \mid n\leq N \}}{N} \\
& \ll \frac{m-n}{q^k} + O\left((q^k)^{1/3} \frac{1}{N^{1/6}} \right)\\
&< \vert I \vert + \delta + O\left((q^k)^{1/3} \frac{1}{N^{1/6}} \right)\\
&= \vert I \vert + \delta + O_\delta \left( \frac{1}{N^{1/6}} \right).
\end{align*}
Since $\delta$ and $I$ were arbitrary, this shows that
\begin{equation*}
D_N ( \{q^n \nu \}) \ll \frac{1}{N^{1/6}}
\end{equation*}
for any $\nu \in E$ and any base $q$.

\subsection{Turing's Algorithm}\label{Turing_sec}

Since Turing's algorithm has been very well studied in \cite{becher_figueira_picchi2007:turing_unpublished}, we restrict ourselves to presenting their result in our terminology.
With respect to the quality of convergence to normality Becher, Figueira and Picchi note (Remark 23 in \cite{becher_figueira_picchi2007:turing_unpublished}) that for each initial segment of $\alpha$ of length $N = k 2^{2n+1}$ expressed to each base up to $e^L$ all words of length up to $L = \sqrt{\log N} / 4$ occur with the expected frequency plus or minus $e^{-L^2}$. Here, $k$ is a positive integer parameter, and $n$ is the step of the algorithm.

The discrepancy of $\{b^n \alpha\}$ for some base $b\geq 2$ can then be calculated as follows. Fix some arbitrary $\epsilon>0$ and an subinterval $I\subset [0,1)$. Let $n$ be large enough, such that $\frac{2}{b^L} < \epsilon$. Approximate $I$ by a $b^L$-adic interval $[\frac{c}{b^L}, \frac{d}{b^L})$ such that $[\frac{c-1}{b^L}, \frac{d+1}{b^L}) \supset I \supset [\frac{c}{b^L}, \frac{d}{b^L})$. Then
\begin{align*}
\frac{\sharp \{0 \leq m < N \mid \{b^m \alpha\} \in I \}}{N}
<& \frac{\sharp \{0 \leq m < N \mid \{b^m \alpha\} \in [\frac{c}{b^L}, \frac{d}{b^L}) \}}{N}\\
\leq &  \frac{d-c+2}{b^L} + O(e^{-L^2})) \\
\leq &  \vert I \vert + \epsilon + O\left(\frac{1}{N^{1/16}}\right).
\end{align*}

Since $I$ and $\epsilon$ were arbitrary this means that $\{b^n \alpha\}$ is uniformly distributed modulo one with discrepancy bounded by $O(\frac{1}{N^{1/16}})$.
\medskip

\subsubsection*{Acknowledgements} This research was supported by the Austrian Science Fund (FWF): I~1751-N26; W1230, Doctoral Program ``Discrete Mathematics''; and  SFB F 5510-N26. I would like to thank Manfred Madritsch and Robert Tichy for many discussions on the subject and for reading several versions of this manuscript.




\end{document}